\documentclass[]{aspm2}

\articleinfo{}{}{}

\newcommand{\R}{\mathbb{R}}

\newcommand{\C}{\mathbb{C}}

\usepackage{verbatim}
\usepackage{amssymb}
\usepackage{amsbsy}
\usepackage{amscd}
\usepackage{amsmath}
\usepackage{amsthm}
\usepackage[mathscr]{eucal}
\usepackage{color}

\newtheorem{theorem}{Theorem}
\newtheorem{defn}{Definition}\numberwithin{defn}{section}
 
 \newtheorem{lem}[defn]{Lemma}

\numberwithin{equation}{section}
\allowdisplaybreaks

\title[Hardy type inequalities]{Remarks on the Hardy type inequalities\\
with remainder terms in the framework of equalities}

\dedicatory{\footnotesize\it Dedicated to Professor Nakao Hayashi on the occasion of his sixtieth birthday} 

\author[S. Machihara]{Shuji Machihara}
\author[T. Ozawa]{Tohru Ozawa}
\author[H. Wadade]{Hidemitsu Wadade}

\address{S. Machihara\,:\,Department of Mathematics, Faculty of Science, Saitama
University, Saitama 338-8570, Japan}
\email{machihar@mail.saitama-u.ac.jp}

\address{T. Ozawa\,:\,Department of Applied Physics, Waseda University, Tokyo 169-8555, Japan}
\email{txozawa@waseda.jp}

\address{H. Wadade\,:\,Faculty of Mechanical Engineering, 
Institute of Science and Engineering, 
Kanazawa University, Kanazawa, 920-1192, Japan}
\email{wadade@se.kanazawa-u.ac.jp}

\rcvdate{}
\rvsdate{}



\begin{document}

\begin{abstract}
We study the Hardy type inequalities in the framework of equalities. 
We present equalities which immediately imply Hardy type inequalities by dropping the remainder term. 
Simultaneously we give a characterization of the class of functions which makes the remainder term vanish. 
A point of our observation is to apply an orthogonality properties in general Hilbert space, 
and which gives a simple and direct understanding 
of the Hardy type inequalities as well as the nonexistence of nontrivial extremizers.
\end{abstract}

\maketitle

\section{Introduction and the main results}\label{section1}
Let $\Omega$ be a domain in $\Bbb R^n$ with $n\geq 3$ and assume $0\in\Omega$.  
The classical Hardy inequality states that the inequality 
\begin{align}\label{hardy-cl}
\left(\frac{n-2}{2}\right)^2\int_\Omega\frac{|f|^2}{|x|^2}dx
\leq\int_\Omega|\nabla f|^2dx
\end{align}
holds for all $f\in H^1_0(\Omega)$, where the constant $\left(\frac{n-2}{2}\right)^2$ is best-possible. 
It is also well-known that the inequality \eqref{hardy-cl} admits no nontrivial extremizers, 
and this fact implies a possibility for  \eqref{hardy-cl} to be improved by adding some remainder terms. 
In fact, the authors in \cite{b-v} proved that  the following improved Hardy inequality 
\begin{align}\label{hardy-bv}
\left(\frac{n-2}{2}\right)^2\int_\Omega\frac{|f|^2}{|x|^2}dx+\Lambda \int_\Omega|f|^2dx
\leq\int_\Omega|\nabla f|^2dx 
\end{align}
holds for all $f\in H^1_0(\Omega)$ provided that $\Omega$ is bounded, where the constant $\Lambda$ in \eqref{hardy-bv} is given by 
$\Lambda=\Lambda(n,\Omega)=z_0^2\omega_n^{\frac{2}{n}}|\Omega|^{-\frac{2}{n}}$, 
and $\omega_n$ and $|\Omega|$ denote the Lebesgue measures of the unit ball and $\Omega$ on $\Bbb R^n$, respectively, 
and the absolute constant $z_0$ denotes the first zero of the Bessel function $J_0(z)$. 
The constant $\Lambda$ is optimal if $\Omega$ is a ball, but still the inequality \eqref{hardy-bv} admits no nontrivial extremizers. 
More generally, the authors in \cite{b-v} obtained the inequality 
\begin{align*}
\left(\frac{n-2}{2}\right)^2\int_\Omega\frac{|f|^2}{|x|^2}dx+\tilde\Lambda\left(\int_\Omega|f|^pdx\right)^{\frac{2}{p}}
\leq\int_\Omega|\nabla f|^2dx
\end{align*}
for $f\in H^1_0(\Omega)$, where $1<p<\frac{2n}{n-2}$ and $\tilde\Lambda$ is a positive constant independent of $u$. 
Similar improvements have been done for the Hardy inequality not only in the $L^2$-setting but in $L^p$-setting with 
some remainder terms, see for instance \cite{b-f-t, b-m, b-m-s, g-g-m, vz}. 

Hardy type inequalities are known as useful mathematical tools in various fields 
such as real analysis, functional  analysis, probability and partial differential equations. 
In fact, Hardy type inequalities and their improvements are applied in many contexts. 
For instance, Hardy type inequalities were utilized in investigating the stability of solutions 
of semi-linear elliptic and parabolic equations in \cite{b-v, cm}. 
As for the existence and asymptotic behavior of solutions of the heat equation involving singular potentials, see \cite{cm2, vz}. 
Among others we refer to \cite{an, b-d, d, fi, kp, maz} for the concrete applications of Hardy type inequalities. 
We also refer to \cite{da, o-k} for a comprehensive understanding of Hardy type inequalities. 

Based on the historical remarks on the Hardy type inequalities, our purpose in this paper is 
to establish the classical Hardy inequalities in the frame work of equalities 
which immediately imply the Hardy inequalities by dropping the remainder terms. 
At the same time, those equalities characterize the form of the vanishing remainder terms. 
Our method on the basis of equalities presumably provides a simple and direct 
understanding of the Hardy type inequalities as well as the nonexistence of nontrivial extremizers.

In what follows, we always assume $\Omega=\Bbb R^n$ and the standard $L^2(\R^n)$ norm is denoted by $\|\cdot\|_{2}$. 
Then the Hardy type inequalities in $L^2$-setting that we discuss in this paper are the 
following:
\begin{align}
\left\|\frac{f}{|x|}\right\|_2
&\le\frac{2}{n-2}\left\|\frac{x}{|x|}\cdot\nabla f\right\|_2, 
\qquad n\ge3, \label{eq1-1}\\
\sup_{R>0}\left\|\frac{f-f_R}{|x|^{\frac{n}{2}}\log\frac{R}{|x|}}\right\|_2
&\le2\left\|\frac{1}{|x|^{\frac{n}{2}-1}}\frac{x}{|x|}\cdot\nabla f\right\|_2, 
\qquad n\ge2, \label{eq1-2}\\
\int_0^{\infty}x^{-p-1}\left|\int_0^xf(y)dy\right|^2dx
&\le\Big(\frac{2}{p}\Big)^2\int_0^{\infty}x^{-p+1}|f(x)|^2dx, \label{eq1-3}\\
\int_0^{\infty}x^{p-1}\left|\int_x^{\infty}f(y)dy\right|^2dx
&\le\Big(\frac{2}{p}\Big)^2\int_0^{\infty}x^{p+1}|f(x)|^2dx, \label{eq1-4}
\end{align}
where $f_R(x)=f\Big(R\frac{x}{|x|}\Big)$ and $p>0$. 
The inequalities \eqref{eq1-1}, \eqref{eq1-3}, and \eqref{eq1-4} are standard 
(see \cite{f} for instance), while \eqref{eq1-2} is rather new 
(see \cite{m-o-w1, m-o-w3}). In addition, as we noticed in \cite{m-o-w3}, the logarithmic Hardy inequality 
\eqref{eq1-2} has a scaling property. 

We state our main theorems. We denote by $\partial_r$ the radial derivative 
defined by $\partial_r=\frac{x}{|x|}\cdot\nabla=\sum_{j=1}^n\frac{x_j}{|x|}\partial_j$. 
The space $D^{1,2}(\R^n)$ denotes the completion of $C_0^\infty(\Bbb R^n)$ 
under the Dirichlet norm $\|\nabla\cdot\|_2$. 
Also the notation $S^{n-1}$ denotes the unit sphere in $\Bbb R^n$ endowed with the Lebesgue measure $\sigma$. 
Our first theorem now reads:

\begin{theorem}\label{thm1}
Let $n\ge3$. Then the equalities
\begin{align}\label{eq1-5}
\left(\frac{n-2}{2}\right)^2\left\|\frac{f}{|x|}\right\|_2^2
&=\|\partial_rf\|_2^2-\left\|\partial_rf+\frac{n-2}{2|x|}f\right\|_2^2 \\
&=\|\partial_rf\|_2^2-\left\||x|^{-\frac{n-2}{2}}\partial_r(|x|^{\frac{n-2}2}f)\right\|_2^2
\label{eq1-6}
\end{align}
hold for all $f\in D^{1,2}(\R^n)$. 
Moreover, the second term in the right hand side of \eqref{eq1-5} or \eqref{eq1-6} vanishes if and only if $f$ takes 
the form 
\begin{equation}\label{eq1-7}
f(x)=|x|^{-\frac{n-2}{2}}\varphi\Big(\frac{x}{|x|}\Big)
\end{equation}
for some function $\varphi:S^{n-1}\to\C$, which makes the left hand side of 
\eqref{eq1-5} infinite unless $\int_{S^{n-1}}|\varphi(\omega)|^2d\sigma(\omega)=0$: 
\begin{equation}\label{eq1-8}
\frac{|f|^2}{|x|^2}=\frac{\big|\varphi\big(\frac{x}{|x|}\big)\big|^2}{|x|^n}
\notin L^1(\R^n). 
\end{equation}
\end{theorem}

We remark that as in \eqref{eq1-8}, functions of the form \eqref{eq1-7} 
imply the nonexistence of nontrivial extremizers for \eqref{eq1-1}. 
The corresponding integral
diverges at both origin and infinity. 
A similar result to Theorem \ref{thm1} can be found in \cite{b-d-k, d-v}. 
However, the essential ideas for the proofs are different. Indeed, the proof in \cite{b-d-k} is done by direct calculations with respect to the quotient with the optimizer of a Hardy type inequality. On the other hand, we shall prove Theorem \ref{thm1} by applying an orthogonality argument in general Hilbert space settings. More precisely, an equality 
\begin{equation}\label{eq1-21}
\left(\frac{n-2}{2}\right)^2\left\|\frac{f}{|x|}\right\|_2^2
=\|\nabla f\|_2^2-\left\|\nabla f+\frac{n-2}{2}\frac{x}{|x|^2}f\right\|_2^2
\end{equation}
has been observed in \cite{b-d-k, d-v}. We should remark that \eqref{eq1-5} and \eqref{eq1-21} are
the same for radially symmetric functions and are not the same for nonradial functions.
In fact, the Dirichlet integral is decomposed into radial and spherical components as
\begin{equation}\notag
\|\nabla f\|_2^2=\|\partial_r f\|_2^2+\sum_{j=1}^n
\left\|\left(\partial_j-\frac{x_j}{|x|}\partial_r\right)f\right\|_2^2.
\end{equation}

Next, we state the logarithmic Hardy type equalities in the critical weighted Sobolev spaces. 

\begin{theorem}\label{thm2}
Let $n\ge2$. Then the equalities
\begin{align}\label{eq1-9}
\frac14\left\|\frac{f-f_R}{|x|^{\frac{n}{2}}\log\frac{R}{|x|}}\right\|_2^2
&=\left\|\frac{1}{|x|^{\frac{n}{2}-1}}\partial_rf\right\|_2^2
-\left\|\frac{1}{|x|^{\frac{n}{2}-1}}\left(\partial_rf
+\frac{f-f_R}{2|x|\log\frac{R}{|x|}}\right)\right\|_2^2 \\
&=\left\|\frac{1}{|x|^{\frac{n}{2}-1}}\partial_rf\right\|_2^2
-\left\|\frac{\big|\log\frac{R}{|x|}\big|^{\frac{1}{2}}}{|x|^{\frac{n}{2}-1}}
\partial_r\left(\frac{f-f_R}{\big|\log\frac{R}{|x|}\big|^{\frac12}}\right)\right\|_2^2
\label{eq1-10}
\end{align}
hold for all $R>0$ and all $f\in L^1_{loc}(\R^n)$ with 
$\frac{1}{|x|^{\frac{n}{2}-1}}\nabla f\in L^2(\R^n)$, where $f_R$ is defined by 
$f_R(x)=f\Big(R\frac{x}{|x|}\Big)$.
Moreover, the second term in the right hand side of \eqref{eq1-9} or \eqref{eq1-10} vanishes if and only if $f-f_R$ takes 
the form 
\begin{equation}\label{eq1-11}
f(x)-f_R(x)=\left|\log\frac{R}{|x|}\right|^{\frac{1}{2}}\varphi\Big(\frac{x}{|x|}\Big)
\end{equation}
for some function $\varphi:S^{n-1}\to\C$, which makes the left hand side of 
\eqref{eq1-9} infinite unless $\int_{S^{n-1}}|\varphi(\omega)|^2d\sigma(\omega)=0$: 
\begin{equation}\label{eq1-12}
\frac{|f-f_R|^2}{|x|^n\Big|\log\frac{R}{|x|}\Big|^2}
=\frac{\big|\varphi\big(\frac{x}{|x|}\big)\big|^2}{|x|^n\Big|\log\frac{R}{|x|}\Big|}
\notin L^1(\R^n). 
\end{equation}
\end{theorem}

As in \eqref{eq1-12}, functions of the form \eqref{eq1-11} 
imply the nonexistence of nontrivial extremizers for \eqref{eq1-2}. The corresponding integral 
diverges at both origin and infinity and, in addition, on the sphere of radius $R>0$. 

The final theorem in this paper is one-dimensional Hardy type equalities stated as follows. 

\begin{theorem}\label{thm3}
Let $n=1$ and $p>0$. Then:
\begin{enumerate}
 \item The equalities
 \begin{align}
 &\label{eq1-13}\left(\frac{p}{2}\right)^2\int_0^{\infty}x^{-p+1}\left|\frac{1}{x}\int_0^xf(y)dy\right|^2dx \\
 &\notag=\int_0^{\infty}x^{-p+1}|f(x)|^2dx
 -\int_0^{\infty}x^{-p+1}\left|f(x)-\frac{p}{2x}\int_0^xf(y)dy\right|^2dx  \\
 &\notag=\int_0^{\infty}x^{-p+1}|f(x)|^2dx
 -\int_0^{\infty}x\left|\frac{d}{dx}\left(x^{-\frac{p}{2}}\int_0^xf(y)dy\right)\right|^2dx 
 \end{align}
 hold for all $f\in L^1_{\text{loc}}(0,\infty)$ with $|x|^{-p+1}|f|^2\in L^1(0,\infty)$. 
 Moreover, the second term in the right hand side of \eqref{eq1-13} vanishes if and only if 
 \begin{equation}\notag
 \int_0^xf(y)dy=cx^{\frac{p}{2}}
 \end{equation}
 for some $c\in\C$, which makes the left hand side of \eqref{eq1-13} infinite unless 
 $c=0$:
 \begin{equation}\notag
 x^{-p+1}\left|\frac{1}{x}\int_0^xf(y)dy\right|^2=|c|^2x^{-1}\notin L^1(0,\infty). 
 \end{equation}
 \item The equalities
 \begin{align}
&\label{eq1-17}\left(\frac{p}{2}\right)^2\int_0^{\infty}x^{p+1}\left|\frac{1}{x}\int_x^{\infty}f(y)dy\right|^2dx \\
 &\notag=\int_0^{\infty}x^{p+1}|f(x)|^2dx
 -\int_0^{\infty}x^{p+1}\left|f(x)-\frac{p}{2x}\int_x^{\infty}f(y)dy\right|^2dx\\
 &\notag=\int_0^{\infty}x^{p+1}|f(x)|^2dx
 -\int_0^{\infty}x\left|\frac{d}{dx}\left(x^{\frac{p}{2}}\int_x^{\infty}f(y)dy\right)\right|^2dx 
 \end{align}
 hold for all $f\in L^1_{\text{loc}}(0,\infty)$ with $|x|^{p+1}|f|^2\in L^1(0,\infty)$. 
 Moreover, the second term in the right hand side of \eqref{eq1-17} vanishes if and only if 
 \begin{equation}\notag
 \int_x^{\infty}f(y)dy=cx^{-\frac{p}{2}}
 \end{equation}
 for some $c\in\C$, which makes the left hand side of \eqref{eq1-17} infinite unless 
 $c=0$:
 \begin{equation}\notag
 x^{p+1}\left|\frac{1}{x}\int_x^{\infty}f(y)dy\right|^2=|c|^2x^{-1}\notin L^1(0,\infty). 
 \end{equation}
\end{enumerate}
\end{theorem}

In the recent paper \cite{i-i-o}, 
the authors actually extended the equalities in Theorem \ref{thm2} with $n=2$ in $L^n(\Bbb R^n)$-settings. However, we shall re-prove Theorem \ref{thm2} in this paper 
in order to compare the sub-critical case, the critical case and the one-dimensional case 
corresponding to Theorem \ref{thm1}, Theorem \ref{thm2} and Theorem \ref{thm3}, respectively. 

The proofs of all theorems in this paper is essentially based on orthogonality 
in general Hilbert space settings. Therefore, there would be a possibility 
to establish other equalities on the functional spaces equipped with 
the Hilbert structure by applying our method. 

Concerning related references to Hardy type inequalities, we also refer to \cite{b-v}, \cite{co}, \cite{d-j-s-j}, \cite{g-g-m} and \cite{mi}. 
Indeed, when a function $f$ is radially symmetric on the unit ball at the origin, 
the corresponding inequality to \eqref{eq1-5} or \eqref{eq1-6} was obtained by \cite[p.454]{b-v}. 
In \cite{g-g-m}, authors established the Hardy inequalities 
with remainder terms on bounded domains by utilizing the argument used in \cite{b-v}. 
The papers \cite{co}, \cite{d-j-s-j} and \cite{mi} proposed simple proofs of classical Hardy, Rellich and Caffarelli-Kohn-Nirenberg inequalities. 
We study the Hardy type inequalities in a different perspective from those papers 
above in the sense that we derive the inequalities from equalities. 
\\

We prove Theorems \ref{thm1} - \ref{thm3} in subsequent sections. 
For simplicity, we prove the theorems for $f\in C^{\infty}_0(\R^n;\C)$. 
The proofs are completed by density (see \cite{m-o-w1, m-o-w3}).
The main idea of the 
proofs is given by the following ``orthogonality lemma."

\begin{lem}\label{lem1}
Let $X$ be a scalar product space with scalar product $(\cdot|\cdot)$. 
Let $c>0$. Then the following statements are equivalent.
\begin{enumerate}
 \item The equality
 \begin{equation}\label{eq1-23}
 \|u\|^2=-2c\text{\rm Re}(u|v)
 \end{equation}
 holds for all $u,v\in X$. 
 \item The equality
 \begin{equation}\label{eq1-24}
 \|u\|^2=4c^2\|v\|^2-\|u+2cv\|^2
 \end{equation}
 holds for all $u,v\in X$. 
\end{enumerate}
\end{lem}

\begin{proof}
The lemma follows by another equivalent equality:
\begin{equation}\label{eq1-25}
\text{Re}(u|u+2cv)=0.
\end{equation}
\end{proof}

The standard proof of Hardy inequality is based on the inequality
\begin{equation}\label{eq1-26}
\|u\|^2\le2c\|u\|\|v\|
\end{equation}
by the Cauchy-Schwarz inequality applied to \eqref{eq1-23}. In this paper, 
we regard \eqref{eq1-23} as the orthogonality relation \eqref{eq1-25} of 
$u$ and $u+2cv$ and then we regard $2cv$ as a difference: 
$2cv=(u+2cv)-u$. The equality \eqref{eq1-24} gives an explicit condition for the case of 
equality in \eqref{eq1-26}. 

There are many papers on the Hardy type inequalities and related subjects. We refer the 
readers to \cite{b-e-h-l}--\cite{t} and references therein.

\section{Proof of Theorem \ref{thm1}}

We introduce polar coodinates $(r,\omega)=(|x|,\frac{x}{|x|})\in(0,\infty)\times S^{n-1}$ 
and we write the integral on the left hand side of \eqref{eq1-5} as
\begin{equation}\notag
\begin{split}
\int_{\R^n}\frac{|f(x)|^2}{|x|^2}dx&=\int_0^{\infty}r^{n-3}
\int_{S^{n-1}}|f(r\omega)|^2d\sigma(\omega)dr \\
&=-\frac{2}{n-2}\text{Re}\int_0^{\infty}r^{n-2}
\int_{S^{n-1}}f(r\omega)\overline{\omega\cdot\nabla f(r\omega)}d\sigma(\omega)dr \\
&=-\frac{2}{n-2}\text{Re}\int_{\R^n}\frac{f(x)}{|x|}\overline{\partial_rf(x)}dx,
\end{split}
\end{equation}
where we have carried out integration by parts in the radial variable. 
By Lemma \ref{lem1} with $c=\frac{1}{n-2}, u=\frac{f}{|x|}$, and $v=\partial_rf$, in 
the Hilbert space $L^2(\R^n)$, we obtain \eqref{eq1-5} as \eqref{eq1-24}. 
By a direct calculation, \eqref{eq1-6} follows from \eqref{eq1-5}. The remainder 
terms vanish if and only if $\partial_r(|x|^{\frac{n-2}{2}}f)=0$, namely, 
$|x|^{\frac{n-2}{2}}f$ is a function on the sphere. This completes the proof of 
Theorem \ref{thm1}. 

\section{Proof of Theorem \ref{thm2}}
Let $R>0$. In the same way as in the proof of Theorem \ref{thm1}, we calculate 
\begin{equation}\notag
\begin{split}
&\int_{|x|<R}\frac{|f(x)-f_R(x)|^2}{|x|^n\Big(\log\frac{R}{|x|}\Big)^2}dx\\
&=\int_0^{R}\frac{1}{r\big(\log\frac{R}{r}\big)^2}
\int_{S^{n-1}}|f(r\omega)-f(R\omega)|^2d\sigma(\omega)dr \\
&=-2\text{Re}\int_0^{R}\frac{1}{\log\frac{R}{r}}
\int_{S^{n-1}}(f(r\omega)-f(R\omega))\overline{\omega\cdot\nabla f(r\omega)}d\sigma(\omega)dr \\
&=-2\text{Re}\int_{|x|<R}\frac{f(x)-f_R(x)}{|x|^{\frac{n}{2}}\log\frac{R}{|x|}}
\frac{\overline{\partial_rf(x)}}{|x|^{\frac{n}{2}-1}}dx,
\end{split}
\end{equation}
where we have carried out integration by parts in $r$ and the contribution on the 
boundary vanishes since
$
\log\frac{R}{r}=\int_1^{\frac{R}{r}}\frac{dt}{t}\ge\frac{\frac{R}{r}-1}{\frac{R}{r}}
=\frac{R-r}{R}\ge0$ and
\begin{align}
|f(r\omega)-f(R\omega)|^2\le\|\nabla f\|_{\infty}^2(R-r)^2. \label{eq3-3}
\end{align}
By Lemma \ref{lem1} with $c=1, u=\frac{f-f_R}{|x|^{\frac{n}{2}}\log\frac{R}{|x|}}$, 
and $v=\frac{\partial_rf}{|x|^{\frac{n}{2}-1}}$, in the Hilbert space 
$L^2(B(0;R))$, where $B(0;R)=\{x\in\R^n;|x|<R\}$, we obtain 
\begin{equation}\label{eq3-4}
\begin{split}
&\frac14\left\|\frac{f-f_R}{|x|^{\frac{n}{2}}\log\frac{R}{|x|}}\right\|^2_{L^2(B(0;R))}
=\left\|\frac{\partial_rf}{|x|^{\frac{n}{2}-1}}\right\|^2_{L^2(B(0;R))}\\
&-\left\|\frac{1}{|x|^{\frac{n}{2}-1}}\left(\partial_r f
+\frac{f-f_R}{2|x|\log\frac{R}{|x|}}\right)\right\|^2_{L^2(B(0;R))}.
\end{split}
\end{equation}
Similarly, we obtain
\begin{equation}\notag
\begin{split}
&
\int_{|x|>R}\frac{|f(x)-f_R(x)|^2}{|x|^n\Big(\log\frac{R}{|x|}\Big)^2}dx\\
&=\int_{R}^{\infty}\frac{1}{r\big(\log\frac{R}{r}\big)^2}
\int_{S^{n-1}}|f(r\omega)-f(R\omega)|^2d\sigma(\omega)dr \\
&=-2\text{Re}\int_{R}^{\infty}\frac{1}{\log\frac{R}{r}}
\int_{S^{n-1}}(f(r\omega)-f(R\omega))\overline{\omega\cdot\nabla f(r\omega)}d\sigma(\omega)dr \\
&=-2\text{Re}\int_{|x|>R}\frac{f(x)-f_R(x)}{|x|^{\frac{n}{2}}\log\frac{R}{|x|}}
\frac{\overline{\partial_rf(x)}}{|x|^{\frac{n}{2}-1}}dx,
\end{split}
\end{equation}
where the contribution on the boundary vanishes due to \eqref{eq3-3} and 
\begin{equation}\notag
\left|\log\frac{R}{r}\right|
=\log\frac{r}{R}
=\int_1^{\frac{r}{R}}\frac{dt}{t}\ge\frac{\frac{r}{R}-1}{\frac{r}{R}}
=\frac{r-R}{r}=\frac{|R-r|}{r}.
\end{equation}
In the same way as in the derivation of \eqref{eq3-4}, we obtain
\begin{equation}\label{eq3-7}
\begin{split}
&\frac14\left\|\frac{f-f_R}{|x|^{\frac{n}{2}}\log\frac{R}{|x|}}\right\|^2
_{L^2(\R^n\backslash\overline{B(0;R)})}
=\left\|\frac{\partial_rf}{|x|^{\frac{n}{2}-1}}\right\|^2_{L^2(\R^n\backslash\overline{B(0;R)})}
\\&-\left\|\frac{1}{|x|^{\frac{n}{2}-1}}\left(\partial_r f
+\frac{f-f_R}{2|x|\log\frac{R}{|x|}}\right)\right\|^2_{L^2(\R^n\backslash\overline{B(0;R)})}.
\end{split}
\end{equation}
Then \eqref{eq1-9} follows by adding both sides of \eqref{eq3-4} and \eqref{eq3-7}. 
By a direct calculation, \eqref{eq1-10} follows from \eqref{eq1-9}, where we notice that 
$\partial_rf_R=0$. The rest of the statements follows in the same way as in the 
proof of Theorem \ref{thm1}. 

\section{Proof of Theorem \ref{thm3}}
The integral on the left hand side of \eqref{eq1-13} is rewritten by integration 
by parts as 
\begin{equation}\notag
\begin{split}
\int_0^{\infty}x^{-p+1}\left|\frac{1}{x}\int_0^xf(y)dy\right|^2dx
&=\int_0^{\infty}x^{-p-1}\left|\int_0^xf(y)dy\right|^2dx \\
&=\frac{2}{p}\text{Re}\int_0^{\infty}x^{-p}\overline{f(x)}\int_0^xf(y)dydx \\
&=\frac{2}{p}\text{Re}\int_0^{\infty}x^{-p+1}\Big(\frac{1}{x}\int_0^xf(y)dy\Big)
\overline{f(x)}dx.
\end{split}
\end{equation}
By Lemma \ref{lem1} with $c=\frac{1}{p}, u=\frac{1}{x}\int_0^xf(y)dy$, and 
$v=-f$, in the Hilbert space $L^2(0,\infty;x^{-p+1}dx)$, we obtain \eqref{eq1-13} as 
\eqref{eq1-24}. 

Similarly, the integral on the left hand side of \eqref{eq1-17} is rewritten as 
\begin{equation}\notag
\begin{split}
\int_0^{\infty}x^{p+1}\left|\frac{1}{x}\int_x^{\infty}f(y)dy\right|^2dx
&=\int_0^{\infty}x^{p-1}\left|\int_x^{\infty}f(y)dy\right|^2dx \\
&=\frac{2}{p}\text{Re}\int_0^{\infty}x^{p}\overline{f(x)}\int_x^{\infty}f(y)dydx \\
&=\frac{2}{p}\text{Re}\int_0^{\infty}x^{p+1}\Big(\frac{1}{x}\int_x^{\infty}f(y)dy\Big)
\overline{f(x)}dx.
\end{split}
\end{equation}
By Lemma \ref{lem1} with $c=\frac{1}{p}, u=\frac{1}{x}\int_x^{\infty}f(y)dy$, and 
$v=-f$, in the Hilbert space $L^2(0,\infty;x^{p+1}dx)$, we obtain \eqref{eq1-17} as 
\eqref{eq1-24}. 

The rest of the statements follows by a direct calculation. 
\\

\begin{center}
{\bf\sc Acknowledgement}
\end{center}

We are grateful to the referee for his/her valuable comments.



\begin{thebibliography}{99}
\bibitem{an}
A. Ancona, 
On strong barriers and an inequality of Hardy for domains in $\Bbb R^n$, 
J. London Math. Soc., {\bf 34} (1986), 274--290. 

\bibitem{b-e-h-l}
A. Balinsky, W. D. Evans, D. Hundertmark and R. T. Lewis, 
On inequalities of Hardy-Sobolev type, 
Banach J. Math. Anal., {\bf 2} (2008), 94--106.

\bibitem{b-f-t}
G. Barbatis, S.Filippas and A. Tertikas, 
A unified approach to improved $L^p$ Hardy inequalities with best constants, 
Trans. Am. Math. Soc., {\bf 356} (2004), 2169--2196.

\bibitem{b}
W. Beckner,  
Pitt's inequality and the fractional Laplacian: sharp error estimates, 
Forum Math., {\bf 24} (2012), 177--209.

\bibitem{b-d}
K. Bogdan and B. Dyda, 
The best constant in a fractional Hardy inequality, 
Math. Nachr., {\bf 284} (2011), 629--638. 

\bibitem{b-d-k}
K. Bogdan, B. Dyda and P. Kim, 
Hardy inequalities and non-explosion results for semigroups, 
Potential Anal., {\bf 44} (2016), 229--247. 

\bibitem{b-m}
H. Brezis and M. Marcus, 
Hardy's inequality revisited, 
Ann. Scuola. Norm. Sup. Pisa, {\bf 25} (1997), 217--237. 

\bibitem{b-m-s}
 H. Brezis, M. Marcus and I. Shafrir, 
Extremal functions for Hardy's inequality with weight, 
J. Funct. Anal., {\bf 171} (2000), 177--191. 

\bibitem{b-v}
H. Brezis and J. L. V{\'a}zquez, 
Blow-up solutions of some nonlinear elliptic problems, 
Rev. Mat. Univ. Complut. Madrid., {\bf 10} (1997), 443--469.

\bibitem{cm2}
X. Cabr\'e and Y. Martel, 
Existence versus explosion instantan\'ee pour des \'equations de la 
chaleur lin\'eaires avec potential singulier, 
C. R. Acad. Sci. Paris S\'er. I, {\bf 329} (1999), 973--978. 

\bibitem{cm}
X. Cabr\'e and Y. Martel, 
Weak eigenfunctions for the linearization of extremal elliptic problems, 
J. Funct. Anal., {\bf 156} (1998), 30--56. 

\bibitem{co}
D. G. Costa,
Some new and short proofs for a class of Caffarelli-Kohn-Nirenberg 
type inequalities, 
J. Math. Anal. Appl., {\bf 337} (2008), 311--317.

\bibitem{da}
E. B. Davies, 
A review of Hardy inequalities, 
Oper. Theory Adv. Appl., {\bf 110} (1999), 55--67. 

\bibitem{d-j-s-j}
Y. Di, L. Jiang, S. Shen and Y. Jin, 
A note on a class of Hardy-Rellich type inequalities, 
J. Inequal. Appl., (2013), 2013:84.

\bibitem{d-v}
J. Dolbeault and B. Volzone, 
Improved Poincar\'e inequalities, 
Nonlinear Anal., {\bf 75} (2012), 5985--6001.

\bibitem{d}
B. Dyda, 
Fractional calculus  for power functions and eigenvalues of the fractional Laplacian, 
Fract. Calc. Appl. Anal., {\bf 15} (2012), 536--555. 

\bibitem{f-t}
S. Filippas and A. Tertikas,
Optimizing improved {H}ardy inequalities,
J. Funct. Anal., {\bf 192} (2002), 186--233.

\bibitem{fi}
P. J. Fitzsimmons, 
Hardy's inequality for Dirichlet forms, 
J. Math. Anal. Appl., {\bf 250} (2000), 548--560. 

\bibitem{f}
G. B. Folland,
``Real analysis," Second edition,
Pure and Applied Mathematics (New York), John Wiley \& Sons, Inc., New York, (1999).

\bibitem{f-s}
R. Frank and R. Seiringer,
Non-linear ground state representations and sharp Hardy inequalities,
J. Funct. Anal., {\bf 255} (2008), 3407--3430.

\bibitem{g-g-m}
F. Gazzola, H. C. Grunau and E. Mitidieri, 
Hardy inequalities with optimal constants and remainder terms, 
Trans. Amer. Math. Soc., {\bf 356} (2003), 2149--2168.

\bibitem{g-m1}
N. Ghoussoub and A. Moradifam, 
Bessel pairs and optimal Hardy and Hardy-Rellich inequalities, 
Math. Ann., {\bf 349} (2011), 1--57.

\bibitem{g-m2}
N. Ghoussoub and A. Moradifam, 
On the best possible remaining term in the Hardy inequality, 
Proc. Natl. Acad. Sci. USA, {\bf 1305} (2008), 13746--13751.

\bibitem{h-l-p}
G. H. Hardy, J. E. Littlewood and G. P\'olya, 
``Inequalities,''
Cambridge University Press, (1952).

\bibitem{i-i}
N. Ioku and M. Ishiwata, 
A scale invariant form of a critical Hardy inequality, 
Int. Math. Res. Not. IMRN, {\bf 18} (2015), 8830--8846. 

\bibitem{i-i-o}
N. Ioku, M. Ishiwata and T. Ozawa, 
Sharp remainder of a critical Hardy inequality, 
Arch. Math. (Basel), {\bf 106} (2016), 65--71.

\bibitem{kp}
A. Kalamajska and K. Pietruska-Paluba, 
On a variant of the Gagliardo-Nirenberg inequality deduced from the Hardy inequality, 
Bull. Pol. Acad. Sci. Math., {\bf 59} (2011), 133--149. 

\bibitem{m-o-w1}
S. Machihara, T. Ozawa and H. Wadade, 
Hardy type inequalities on balls, 
Tohoku Math. J. (2), {\bf 65} (2013), 321--330.

\bibitem{m-o-w2}
S. Machihara, T. Ozawa and H. Wadade, 
Generalizations of the logarithmic Hardy inequality in critical Sobolev-Lorentz spaces, 
J. Inequal. Appl., (2013), 2013:381. 

\bibitem{m-o-w3}
S. Machihara, T. Ozawa and H. Wadade, 
Scaling invariant Hardy inequalities of multiple logarithmic type on the whole space, 
J. Inequal. Appl., (2015), 2015:281.

\bibitem{maz}
V. Maz'ya, 
``Sobolev spaces with applications to elliptic partial differential equations,''
Springer, Heidelberg, (2011). 

\bibitem{mi}
E. Mitidieri, 
A simple approach to Hardy inequalities, 
Mathematical notes, {\bf 67} (2000), 563--572.

\bibitem{o-k}
B. Opic and A. Kufner, 
``Hardy-type inequalities,''
Longman Scientific \& Technical, (1990). 

\bibitem{t}
F. Takahashi, 
A simple proof of {H}ardy's inequality in a limiting case, 
Arch. Math. (Basel)., {\bf 104} (2015), 77--82.

\bibitem{vz}
J. L. Vazquez and E. Zuazua, 
The Hardy inequality and the asymptotic behavior of the heat equation with and inverse-square potential, 
J. Funct. Anal., {\bf 173} (2000), 103--153. 

\end{thebibliography}
\end{document}